\documentclass[11pt]{amsart}
\usepackage{amsmath, amsrefs, amsthm, amssymb}
\usepackage{graphicx}
\usepackage[active]{srcltx}		
\newcommand{\dd}{\,{\rm d}}

\newcommand\R{{\mathbb{R}}}

\newcommand\T{{\mathbb{T}}}

\newcommand\rodnl{\Bigl(\frac{3-\gamma}{2}u^2+\frac{\gamma}{2}u_x^2\Bigr)}

\newtheorem{theorem}{Theorem}[section]
\newtheorem{proposition}[theorem]{Proposition}
\newtheorem{lemma}[theorem]{Lemma}
\newtheorem{corollary}[theorem]{Corollary}

\theoremstyle{definition}

\theoremstyle{remark}
\newtheorem{remark}[theorem]{Remark}

\numberwithin{equation}{section}
\numberwithin{figure}{section}

\begin{document}

\title[Local blowup criteria]{Local-in-space criteria for blowup in shallow water and dispersive rod equations}

\author{Lorenzo Brandolese}

\address{L. Brandolese: Universit\'e de Lyon~; Universit\'e Lyon 1~;
CNRS UMR 5208 Institut Camille Jordan,
43 bd. du 11 novembre,
Villeurbanne Cedex F-69622, France.}
\email{brandolese{@}math.univ-lyon1.fr}
\urladdr{http://math.univ-lyon1.fr/$\sim$brandolese}





\begin{abstract}
We unify a few of the best known results on  wave breaking for the Camassa--Holm equation (by R. Camassa, A.~Constantin, J.~Escher, L. Holm, J. Hyman and others)  in a single  theorem: a sufficient condition for the breakdown is that $u_0'+|u_0|$ is strictly negative in at least one point $x_0\in\R$. Such blowup  criterion looks more natural than the previous ones, as the condition on the initial data is purely local in the space variable. Our method relies on the introduction of two families of Lyapunov functions. Contrary to McKean's necessary and sufficient condition for blowup, our approach  applies to other equations that are not integrable: we illustrate this fact by establishing new local-in-space blowup criteria for an equation modeling nonlinear dispersive waves in elastic rods.
\end{abstract}

\maketitle

\section{Introduction}

\label{sec:localblow}
For many evolution equations, typical well-posedness results hold a priori only for short time intervals.
A basic problem is to establish whether such intervals can be taken of arbitrary length.
When it is not the case, one expects to find a maximal existence time $T^*<\infty$ and some spatial norms of the solution such that $\|u(t,\cdot)\|$  is finite for $t\in(0,T^*)$ and becomes unbounded as $t\uparrow T^*$.
A blowup criterion is nothing but a condition on the initial data $u_0(x)$ implying the latter scenario.
In this paper we will focus on blowup criteria of a very special nature: \emph{local-in-space blowup criteria}.
This notion requires the introduction of an equivalence relation in the class of initial data.

Let $x_0\in \Omega$, where $\Omega$ is the spatial domain where the PDE is considered. We say  that two initial data are ``equivalent at $x_0$'' if they are identical in a neighborhood of $x_0$.
A \emph{local blowup criterion at $x_0$} is a condition on the equivalence class $[u_0]$, implying the finite time blowup of \emph{any solution} arising from initial data equivalent to~$u_0$.

For evolution equations involving  diffusive phenomena, blowup criteria are typically non local with respect to the space variable: indeed, it is usually possible to prevent the blowup of the solutions by perturbating the initial data only in some regions, leaving the data unchanged in the other regions. 

For non-diffusive equations the situation is different.
The simplest example is provided by the inviscid Burgers equation on the line, $u_t+uu_x=0$. Applying the method of the characteristics shows that the condition for the formation of shock waves is that $u_0$ has a negative
derivative in some point~$x_0\in\R$.

The goal of this paper is to establish the analogue of such elementary fact for more realistic models arising in one dimensional shallow water theory and for equations modeling the propagation of  dispersive waves in elastic rods.
Such models are well-posed in Sobolev spaces $H^s(\R)$, at least during a short time interval, for sufficiently large  $s>0$. Our blowup criterion will simply read
\begin{equation} 
\label{BRE}
\inf_{\R}\bigl(u_0'+\beta_\gamma |u_0|\bigr)<0,
\end{equation}
where $\beta_\gamma$ is a constant depending on the physical parameter $\gamma$ (related to the Finger deformation tensor) of the model.
For the celebrated Camassa--Holm equation, such criterion boils down to
\begin{equation}
\label{BCH-inf}
\inf_{\R}\bigl(u_0'+ |u_0|\bigr)<0.
\end{equation}

In the last 20 years, several hundreds of papers were devoted to study of the Camassa--Holm equation or its generalizations, and many of them addressed the issue of blowup (or wave breaking).
At best of our knowledge,  criteria~\eqref{BRE}-\eqref{BCH-inf} remained unnoticed, despite they were anticipated by many earlier results. In fact, we unify in this way some of the best known blowup criteria,
including those established by  R.~Camassa, A.~Constantin, J.~Escher, L.~Holm, J.~Hyman, Y.~Zhou, etc.).
See Section~\ref{refsCH} for a short survey of previous results and the relevant references.

A specific feature of the blowup criteria~\eqref{BRE}-\eqref{BCH-inf} is that they are local
 (in the sense of the previous definition)  at the point $x_0$ where the negative infimum is achieved. This means that it is impossible to prevent the blowup without modifying  $u_0$ \emph{around the point $x_0$}: perturbing $u_0$ only outside a neighborhood of~$x_0$ may  just help in delaying the formation of singularities.

A related  physical interpretation is that fast oscillations will always lead to a blowup for such models (i.e., to the formation of a breaking wave or to the breaking of the rod in finite time), no matter  how small is the region where the oscillations are present and how small are their amplitude.

The organization of the paper is straightforward. There is only one new theorem (Theorem~\ref{th:blow}).
We state it and compare it with previously known results in Section~2. Section~3 is devoted to its proof and Section~4 to technical remarks  and concluding observations.

\section{The main result}

\subsection{The compressible hyper-elastic rod equation}

The propagation of nonlinear waves inside cylindrical hyper-elastic rods, assuming that the diameter is small when compared to the axial length scale, is described by the one dimensional equation
\[
v_\tau+\sigma_1 v v_\xi+\sigma_2 v_{\xi\xi\tau}+\sigma_3(2v_\xi v_{\xi\xi}+vv_{\xi\xi\xi})=0,
\qquad
\xi\in\R, \; \tau>0.
\]
Such equation was derived by H.H.~Dai \cite{Dai98}.
Here $v(\tau,\xi)$ represents the radial stretch relative to a prestressed state, $\sigma_1\not=0$, $\sigma_2<0$ and
$\sigma_3\le0$ are physical constants depending by the material.
The scaling transformations
\[
\tau=\frac{3\sqrt{-\sigma_2}}{\sigma_1}t,\qquad
\xi=\sqrt{-\sigma_2}x,
\]
with $\gamma=3\sigma_3/(\sigma_1\sigma_2)$ and $u(t,x)=v(\tau,\xi)$,
allow us to reduce the above equation to
\begin{equation}
\label{rod-pde}
u_t-u_{xxt}+3uu_x=\gamma(2u_xu_{xx}+uu_{xxx}),
\qquad x\in\R,\;t>0.
\end{equation}

We complement this equation with vanishing boundary conditions at $\pm\infty$: such boundary conditions will be taken into account through an appropriate choice of the functional setting, guaranteeing the well-posedness of the Cauchy problem.
The existence and the orbital stability of solitary waves for equation~\eqref{rod-pde}
is discussed {\it e.g.\/} in~\cite{ConStra00}.
In the case $\gamma=1$, solitary waves are peaked solitons.

For $\gamma=0$, the rod equation~\eqref{rod-pde} reduces to  the well-known BBM equation~\cite{BBM72}, modeling surface waves in a canal. In this case the solutions exist globally, meaning that the BBM equation can model permanent waves, but is unsuitable for describing breaking waves.

For $\gamma=1$~\eqref{rod-pde} becomes the Camassa--Holm equation (CH), modeling long waves in shallow water. Such equation marked an important development in nonlinear dynamics and is of great current interest~\cite{CamHol93}, \cite{AConL09}, \cite{GlaSu}.
It admits strong solutions that exist globally and others that lead to a wave breaking in finite time.
In this case equation~\eqref{rod-pde} has a bi-hamiltonian structure and is integrable. 
The Camassa--Holm equation is thus much better understood than equation~\eqref{rod-pde}.
Useful survey papers on such equation are~\cite{ConstEDP} and~\cite{Mol04}.

We denote by  
\[ 
p(x)=\textstyle\frac{1}{2}e^{-|x|}
\]
the fundamental solution of the operator $1-\partial_x^2$.
Let $y=u-u_{xx}$ be the potential of~$u$. We thus have $u=p*y$, and
$y$ satisfies
\begin{equation*}
y_t+\gamma y_x u+2\gamma y u_x+3(1-\gamma)uu_x=0.
\end{equation*}

It is also convenient to rewrite the Cauchy problem associated with equation~\eqref{rod-pde} 
in the following weak form:
\begin{equation}
\label{rod}
\begin{cases}
u_t+\gamma uu_x=-\partial_x p*\biggl(\displaystyle\frac{3-\gamma}{2}\,u^2+\frac{\gamma}{2}\,u_x^2\biggr), &\;t\in(0,T),\; x\in\R,\\
u(0,x)=u_0(x).
\end{cases}
\end{equation}

For any $\gamma\in\R$, the Cauchy problem for the rod equation is locally well-posed in $H^s$, when $s>3/2$.
More precisely, if $u_0\in H^s(\R)$, $s>3/2$, then there exists a maximal time $0<T^*\le\infty$
and a unique solution $u\in C([0,T^*),H^s)\cap C^1([0,T^*),H^{s-1})$.
Moreover, the solution~$u$ depends continuously on the initial data. 

It is also known that~$u$ admits the invariants
\[
E(u)=\int_\R (u^2+u^2_x)\,dx
\]
and
\[
F(u)=\int_\R (u^3+\gamma uu_x^2)\,dx.
\]
In particular, the invariance of the Sobolev $H^1$-norm of the solution implies that $u(x,t)$ remains uniformly bounded up
to the time~$T^*$.
On the other hand, if $T^*<\infty$ then $\limsup_{t\to T}\|u(t)\|_{H^s}=\infty$ ($s>3/2$) and the precise blowup scenario
of strong solutions is the following:
\begin{equation}
\label{blowup-scena}
T^*<\infty \iff \liminf_{t\to T^*}\Bigl(\inf_{x\in\R}\gamma u_x(t,x)\Bigr)=-\infty.
\end{equation}
Applying the above result with $\gamma=0$ one recovers the fact that solutions of the BBM do not blowup.

The above results are due to Constantin and Strauss \cite{ConStra00}. They were previously established in the
particular case $\gamma=1$ corresponding to the Camassa--Holm equations  (see e.g. \cite{ConEschActa}).

The main result of the present paper is the following blowup criterion for equation~\eqref{rod},
in the case $1\le \gamma\le4$ (for $\gamma\not\in[1,4]$, see Section~\ref{sec-erod}).

\begin{theorem}
\label{th:blow}
Let $1\le \gamma\le4$. There exists a constant $\beta_\gamma\ge0$ (given by the explicit formula~\eqref{def:beta} below, see  Figure~\ref{figure1}) with the following property: 
Let  $T^*$ be the maximal time of the unique solution $u$ of equation~\eqref{rod} in $C([0,T^*),H^s)\cap C^1([0,T^*),H^{s-1})$   arising from $u_0\in H^s(\R)$, with $s>3/2$.
Assume that there exists $x_0\in\R$ such that 
\[
 u_0'(x_0)<-\beta_\gamma|u_0(x_0)|,
\]
then $T^*<\infty$.
\end{theorem}

\begin{remark}
It was observed in~\cite{ConStra00} that, for $\gamma=3$, all non-zero solutions develop a singularity in finite time. This conclusion agrees with Theorem~\ref{th:blow}, since  $\beta_3=0$ and any non-zero initial data 
$u_0\in H^s$
must have a strictly negative derivative in some point.
In fact, the function $\gamma\mapsto\beta_\gamma$, defined for $\gamma\in[1,4]$, 
is continuous,  strictly decreasing on~$[1,3]$ and increasing on $[3,4]$.  Moreover, $\beta_1=1$, $\beta_3=0$
and $\beta_{4}=\frac{1}{2}$.
\end{remark}

Let us restate explicitly the result corresponding to the wave breaking for the Camassa--Holm equation 
$(\gamma=1$):

\begin{corollary}
\label{co:blow} 
Let  $T^*$ the maximal time of the unique solution~$u\in C([0,T^*),H^s)\cap C^1([0,T^*),H^{s-1})$  of the
Camassa--Holm equation on~$\R$,
\begin{equation}
\label{eq-ch}
u_t+ uu_x=-\partial_x p*\biggl(u^2+\frac{1}{2}u_x^2\biggr),
\end{equation}
starting from $u_0\in H^s(\R)$, with $s>3/2$. 
If there exists $x_0\in\R$ such that
\[
  u_0'(x_0)<-|u_0(x_0)|,
\]
then $T^*<\infty$.
\end{corollary}

\begin{figure}
\label{figure1}
\includegraphics[width=10cm,height=7cm]{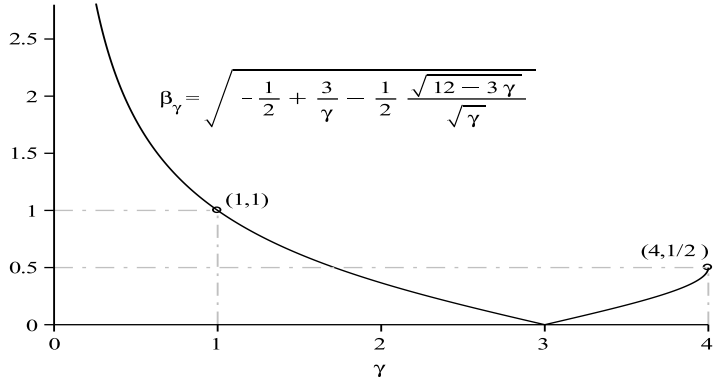}{\centering}
\caption{The plot of $\beta_\gamma$}
\end{figure}

\begin{remark}[The periodic case]
The above results extend to periodic solutions, with the same restrictions on~$\gamma$ and the same constant $\beta_\gamma$:
if $u_0\in H^s(\T)$ with $s>3/2$ and 
\begin{equation}
\label{torus-bu}
\inf_{x\in\T}(u_0'+\beta_\gamma|u_0|)(x)<0,
\end{equation}
then the corresponding solution defined on the torus blows up.
In particular, in the case of the periodic Camassa--Holm equation, wave breaking occurs as soon as 
 $u_0'(x_0)<-|u_0(x_0)|$ at some point $x_0\in \T$.
However, contrary to the case of the whole real line, we expect that some improvements on the expression of~$\beta_\gamma$ should be possible in the case of the torus: finding the best coefficient~$\beta_\gamma$ would require combining ideas of the present paper with the variational techniques used {\it e.g.\/} in~\cite{Wah-NoDEA07}.
\end{remark}

In this paper we will not discuss the continuation of the solution 
after the formation of a singularity. In general, there are several way to continue  solutions
beyond the wave breaking and it is possible to obtain weak solutions that conserve or that dissipate
the energy $\|u\|_{H^1}$. This issue is extensively studied, {\it e.g.\/}  in \cite{BreCon07}, 
\cite{BreCon07bis}, \cite{CocHK05}, \cite{HolJDE07}, \cite{HolJDE07bis}, \cite{HolJDE09}.
On the other hand, we know of no global existence results for strong solutions to equation~\eqref{rod}, excepted for the particular cases $\gamma=0$ or $\gamma=1$ (see, {\it e.g.\/}, \cite{BBM72}, \cite{ConstAIF00}, \cite{McKean04}) and for the examples of smooth solitary waves constructed in~\cite{Dai-Huo} for $0<\gamma<1$.

\subsection{Earlier results on wave breaking for Camassa--Holm}
\label{refsCH}

In their pioneering paper~\cite{CamHol93}, R.~Camassa and L.~Holm announced that if $u_0'(x_0)$ is negative with
$|u'_0(x)|$  sufficiently large, then the solution of the Camassa--Holm equation must lead to a wave breaking.
Several subsequent theorems confirmed their claim. The best known have been nicely surveyed by L. Molinet~\cite{Mol04} and are the following. In the theorems below, the word ``solution'' must be understood as in Corollary~\ref{co:blow}. In the first two theorems the  assumption $u_0\in H^3(\R)$ could be relaxed to $u_0\in H^s(\R)$, $s>3/2$, as noted in \cite{Mol04}.  

\begin{theorem}[Camassa, Holm, Hyman \cite{CamHolHym94}. The more precise formulation below is taken from~\cite{ConEschPisa}]
\label{th:CHH}
If $u_0\in H^3(\R)$, $u_0$ is \emph{odd} and $u_0'(0)<0$, then the corresponding solution of the Camassa--Holm equation blows up in finite time.
\end{theorem}

The proof of the above theorem relies on the fact that the solution remains odd for all time, 
implying that $x\mapsto u_{xx}(t,x)$ is also odd and thus $u_{t,xx}(t,0)=0$ for all~$t\in[0,T)$.
One then deduces the differential inequality $\frac{\dd }{\dd t} [u_x(t,0)] \le -\frac{1}{2} u_x(t,0)^2$, leading
to the condition of blowup scenario. 

The next theorem has been systematically adapted to the many different generalizations of the Camassa--Holm equation.
It was anticipated by~\cite{CamHol93}, \cite{CamHolHym94}, where one finds the idea of obtaining a differential inequality similar to the previous one for $m(t)=u_x(t,\bar x(t))$, where $\bar x(t)\in\R$ is a point where $u_x(\cdot,t)$ attains its minimum. The main technical difficulty consists in proving that the map $t\mapsto m(t)$ is a.e. differentiable.
This subtle point was successfully addressed by A.~Constantin and J.~Escher. The resulting criterion reads as follows:

\begin{theorem}[Constantin, Escher  \cite{ConEschActa}]
\label{th:CE}
Let $u_0\in H^3(\R)$ be such that, at some point $x_0\in\R$, $u'(x_0)<-\|u_0\|_{H^1}/\sqrt2$. Then the solution of the Camassa--Holm equation blows up in finite time.
\end{theorem}

A third interesting blow up criterion is provided by the next theorem.

\begin{theorem}[Constantin, \cite{ConstAIF00}]
\label{th:C}
Let $u_0\in H^3(\R)$ be such that the associated potential $y_0=u_0-(u_0)_{xx}$ changes sign
and satisfies, for some $x_0\in\R$, $y(x)\ge0$ for $x\le x_0$ and $y_0(x)\le0$ for $x\ge x_0$.
Then the solution of the Camassa--Holm equation blows up in finite time.
\end{theorem}

Notice that these three criteria are \emph{non-local} in the sense
explained in the introduction. 
Moreover, despite a few common features, none of them implies the other two.
On the other hand, Corollary~\ref{co:blow} clearly unifies the three above results:
The blowup criterion of Camassa, Holm and Hyman is considerably relaxed: for example, one can replace the antisymmetry condition
on~$u_0$ by the simpler requirement $u_0(0)=0$.
On the other hand, it follows from Corollary~\ref{co:blow} that solutions arising from small perturbations of a nonzero antisymmetric initial datum always blow up.

The criterion by Constantin and Escher follows immediately from the Sobolev embedding inequality
 \[
 \|u\|_{L^{\infty}}\le  \|u_0\|_{H^1}/\sqrt2. 
 \]
The blowup  criterion of Constantin~\cite{ConstAIF00} is deduced from Corollary~\ref{co:blow} applying the elementary identities
(just apply integration by parts):
\begin{equation}
\label{u-u'}
(u_0-u_0')(x)=e^{-x}\int_{-\infty}^x e^\xi y_0(\xi)\dd\xi,
\end{equation}
\begin{equation}
\label{u+u'}
(u_0+u_0')(x)=e^x\int_x^{+\infty} e^{-\xi} y_0(\xi)\dd\xi.
\end{equation}

There is a fourth  blowup criterion, due to Y. Zhou \cite{Zhou2004}, that slightly extends Constantin's theorem.
Zhou's theorem affirms that the solution of the Camassa--Holm equation blows up provided
that there exists $x_1\in\R$ such that 
\[ 
y(x_1)=0,\qquad
\int_{-\infty}^{x_1} e^\xi y_0(\xi)\dd\xi>0\qquad
\text{and}
\qquad
\int_{x_1}^{+\infty} e^{-\xi} y_0(\xi)\dd\xi<0.\]
Using identities~\eqref{u-u'}-\eqref{u+u'} shows that Corollary~\ref{th:blow} encompasses also Zhou's criterion.
In fact, the corollary essentially tells us that  the above restriction $y(x_1)=0$ can be dropped.

On the other hand, using some properties that are specific to the case $\gamma=1$ (in particular the fact that 
the Camassa--Holm equation has a bi-hamiltonian structure and  can be integrated), H. McKean succeded in providing a necessary and sufficient condition for wave-breaking \cite{McKean04}.  His theorem  asserts that,
if $u_0\in H^3(\R)$ and $y_0\in L^1(\R)$, then the solution of the Camassa--Holm equation will develop a singularity in finite time \emph{if and only if} $y_0$ satisfies the following sign condition:
\[
\exists\, x_1<x_2\quad\text{such that}\quad y_0(x_1)>0>y_0(x_2).
\]

Notice that the assumptions of Corollary~\ref{co:blow} are slighlty more general  than that of McKean's theorem.
However, the conclusion of Corollary~\ref{co:blow} is less precise, because it provides a sufficient condition that is no longer necessary. On the other hand, the crucial advantage of our approach is that it makes use of few properties of the Camassa--Holm equation. For this reason our theorem, unlike McKean's, remains valid also in the case  $\gamma>1$. Moreover, it seems hopeful that the proof of Theorem~\ref{th:blow} can be adapted to a wide class of physically relevant equations arising in shallow water theory. We will briefly discuss a few possible extensions in the last section.

\subsection{Main idea of the proof}

One first reduces to the case $u_0\in H^3(\R)$ by approximation. 
As in~\cite{ConstAIF00}, \cite{McKean98},
the starting point is the analysis of the flow map $q(t,x)$, defined by 
\begin{equation}
\label{flow}
\begin{cases}
q_t(t,x)=\gamma u(t,q(t,x)), &t\in(0,T),\; x\in\R,\\
q(0,x)=x.
\end{cases}
\end{equation}
One easily checks that $q\in C^1([0,T)\times\R,\R)$ and $q_x(t,x)>0$ for all~$t\in(0,T)$ and $x\in\R$. 
In the case $\gamma=1$, the potential $y=u-u_{xx}$ satisfies the fundamental identity
\[
 y(t,q(t,x))\bigl(q_x(t,x)\bigr)^2=y_0(x), \qquad t\in[0,T),\quad x\in\R \qquad(\gamma=1),
\]
implying that the zeros and the sign of $y$ are invariant under the flow.
This nice property of~$y$ is an essential feature of the Camassa--Holm.
It implies a few remarkable algebraic identities~(see~\cite{McKean98}),
and can be used not only to obtain wave breaking criteria, but also global existence results, \cite{ConEschActa}, \cite{ConstAIF00}, \cite{McKean04}.

For $\gamma\not=1$ the zeros and the sign of $y$ are no longer invariant of the flow.  
This explains why global existence results are much more difficult to establish in this case.
Indeed, the above identity extends as follows (for $t\in[0,T)$ and $x\in\R$)
\begin{equation*}
y(t,q(t,x))\bigl(q_x(t,x)\bigr)^2=y_0(x)
+3(\gamma-1)\Bigl(\int_0^t (uu_x)(s,q(s,x))q_x(s,x)^2\dd s \Bigr). 
\end{equation*}
Such formula is obtained from the previous one by applying the ``variation of the constants". It can also be  checked directly by computing the total time derivative.
But it turns out that this identity is not as useful as for  $\gamma=1$.

The essential idea of the proof of Theorem~\ref{th:blow} is then to find two constants $\alpha$ and $\beta$
such that, for all $x\in\R$, the functions 
\[t\mapsto e^{\alpha q(t,x)}(\beta u-u_x)(t,q(t,x))\]
are monotone non-decreasing on $[0,T)$, and
the functions 
\[t\mapsto e^{-\alpha q(t,x)}(\beta u+u_x)(t,q(t,x))\]
are non-increasing.
Such monotonicity properties will be convenient substitutes of McKean's algebraic identities~\cite{McKean98}
established for the Camassa--Holm equation.


\section{Two families of Lyapunov functions}

\begin{proof}[Proof of Theorem~\ref{th:blow}]
Let us first consider the case $u_0\in H^3(\R)$. As in \cite{ConstAIF00}, we shall look for a differential inequality for $\frac{\dd}{\dd t}u_x(t,q(t,x))$.
Recalling that $\partial^2_xp*f=p*f-f$, differentiating equation~\eqref{rod} with respect to the~$x$ variable
yields
\begin{equation}
\label{eq:rodx}
u_{tx}+\gamma uu_{xx}=\frac{3-\gamma}{2}u^2-\frac{\gamma}{2}u_x^2-p*\rodnl.
\end{equation}

For $0\le\gamma\le4$, let $\delta=\delta(\gamma)$ given by
\begin{equation}
\label{def:delta}
\delta=\frac{\sqrt\gamma}{4}\bigl(\sqrt{12-3\gamma}-\sqrt \gamma\bigr).
\end{equation}
Notice that for $0\le\gamma\le3$, then $\delta\ge0$ was characterized by Y.~Zhou~\cite{Zhou-Math-Nachr} as the best constant satisfying the inequality
\[
  p*\rodnl\ge \delta u^2.
\]
We will extend this inequality also to the case $3\le \gamma\le 4$ (in this case $\delta\le0$). 
In fact,  the next lemma proves more than this:
\begin{lemma}
\label{lem:zhou}
Let $0\le\gamma\le4$, $0\le\beta\le1$ and $\delta$ as in~\eqref{def:delta}.
Then
\begin{equation}
\label{est:rod1}
(p\pm\beta\partial_xp)*\rodnl\ge\delta u^2.
\end{equation}
\end{lemma}

\begin{proof}
We denote by ${\bf 1}_{\R^+}$ and ${\bf 1}_{\R^-}$ the characteristic functions of $\R^+$ and $\R^-$ respectively. Let $a\in\R$:
 \begin{equation*}
 \begin{split}
 (p{\bf 1}_{\R^+})*(a^2u^2+u_x^2)(x)
 &=\frac{e^{-x}}{2}\int_{-\infty}^x e^\xi(a^2u^2+u_x^2)(\xi)\dd\xi\\
 &\ge a e^{-x}\int_{-\infty}^x e^\xi uu_x\dd\xi\\
 &=\frac{a u^2(x)}{2}-\frac{a e^{-x}}{2}\int_{-\infty}^x e^\xi u^2\dd\xi\\
 &=\frac{a u^2(x)}{2}-a(p{\bf 1}_{\R^+})*(u^2)(x).
 \end{split}
 \end{equation*}
This leads to
\[
(p{\bf 1}_{\R^+})*\bigl((a^2+a)u^2+u_x^2\bigr)\ge \frac{a}{2} u^2.
\]
Choosing $a$ to be the largest real root of the second order equation ($a$ will be negative for $3<\gamma\le4$)
\[
a^2+a={(3-\gamma)}/{\gamma}
\]
we get $\gamma a=2\delta$, hence 
\[
 (p{\bf 1}_{\R^+})*\rodnl\ge \textstyle\frac{\delta}{2} u^2.
\]
The same computations also show that
\[
 (p{\bf 1}_{\R^-})*\rodnl\ge \textstyle\frac{\delta}{2} u^2.
\]
We have, both in the a.e. and the distributional sense,
\[
 \begin{split}
  p-\beta\partial_x p&=(1-\beta)p{\bf 1}_{\R^-}+(1+\beta)p{\bf 1}_{\R^+},\\
 p+\beta\partial_x p&=(1+\beta)p{\bf 1}_{\R^-}+(1-\beta)p{\bf 1}_{\R^+}.
\end{split}
\]
For $0\le \beta\le1$, taking the linear combination in the two last inequalities implies
estimate~\eqref{est:rod1}.
\end{proof}

The optimality of the constant~$\delta$ in inequality~\eqref{est:rod1} follows by the fact that
 for $0<\gamma<3$, the equality holds  at the origin
when $a u=u_x$ in $(-\infty,0)$ and $-a u=u_x$ in $(0,\infty)$.
Here $a$ is the positive root of $a^2+a={(3-\gamma)}/{\gamma}$.
In other words, the choice $u(x)=e^{-a|x|}$ gives the equality in \eqref{est:rod1}
for $x=0$. When  $3<\gamma\le4$, then $a<0$ and one chooses $u(x)=e^{a|x|}$.

\medskip
Let $0<T\le T^*$.
Recalling that $u\in C^1([0,T),H^2)$, we see 
that $u$ and $u_x$ are continuous on $[0,T)\times\R$ and $x\mapsto u(t,x)$ is Lipschitz, uniformly
with respect to~$t$ in any compact time interval in $[0,T)$.
Then the flow map $q(t,x)$ is well defined by~\eqref{flow} in the whole time interval $[0,T)$
and $q\in C^1([0,T)\times\R,\R)$.

We have, for $0<\gamma\le4$,
\begin{equation}
\label{inneq}
 \begin{split}
 \frac{\dd }{\dd t}\bigl[u_x(t,q(t,x))\bigr]
 &=\bigl[u_{tx}+\gamma u u_{xx}\bigr](t,q(t,x))\\
 &=\textstyle\frac{3-\gamma}{2}u^2-\frac{\gamma}{2}u_x^2-p*\rodnl\\
 &\le \Bigl((\textstyle\frac{3-\gamma}{2}-\delta) u^2-\frac{\gamma}{2}u_x^2\Bigr)(t,q(t,x))\\
 &=\textstyle\frac{\gamma}{2}\bigl(\beta_\gamma^2 u^2-u_x^2\bigr)(t,q(t,x)),
 \end{split}
\end{equation}
where
\[
\beta_\gamma^2=\frac{3-\gamma}{\gamma}-\frac{2\delta}{\gamma}. 
\]
According to~\eqref{def:delta}, then we set
\begin{equation}
 \label{def:beta}
 \beta_\gamma=\biggl(
 -\frac{1}{2}+\frac{3}{\gamma}-\frac{\sqrt{12-3\gamma}}{2\sqrt \gamma}
 \biggr)^{1/2}.
\end{equation}
This expression shows that $\gamma\mapsto\beta_\gamma$ is continuous,  decreasing
on $(0,3]$, increasing for $\gamma\in[3,4]$. Moreover, $\beta_1=1$ and $\beta_3=0$ and $\beta_4=\frac{1}{2}$.

The next step is to find a good factorization of $(\beta_\gamma^2 u^2-u_x^2)(t,q(t,x))$ ensuring some monotonicity properties for each factor.
The obvious factorization $\beta_\gamma^2 u^2-u_x^2=(\beta_\gamma u-u_x)(\beta_\gamma u+u_x)$
is not the most interesting one, as it will be revealed by the the next proposition. 

This leads us to study the functions of the form:
\[
A(t,x)=e^{\alpha q(t,x)}(\beta u-u_x)(t,q(t,x)),
\]
and 
\[
B(t,x)=e^{-\alpha q(t,x)}(\beta u+u_x)(t,q(t,x)).
\]
where $0<\gamma\le4$, and $\alpha$ and $\beta$ are real constants to be chosen later.

Computing the derivatives with respect to~$t$ using the definition of the flow map~\eqref{flow} gives, for $\alpha\in\R$ and $0\le\beta\le1$,
\[
 \begin{split}
  A_t(t,x)
  &=e^{\alpha q(t,x)}
    \Bigl[
    \alpha\beta\gamma u^2-\alpha\gamma u u_x+\beta(u_t+\gamma uu_x)-(u_{tx}+\gamma uu_{xx})
    \Bigr]\\
   &=e^{\alpha q(t,x)}
   \Bigl[\Bigl(\alpha\beta\gamma-\textstyle\frac{3-\gamma}{2}\Bigl)u^2+\frac{\gamma}{2}u_x^2-\alpha\gamma uu_x +(p-\beta\partial_x p)*\rodnl
   \Bigr]\\
  &\ge e^{\alpha q(t,x)}
  \Bigl[\Bigl(\alpha\beta\gamma-\textstyle\frac{3-\gamma}{2}+\delta\Bigl)u^2
  -\alpha\gamma uu_x+\frac{\gamma}{2}u_x^2
   \Bigr].
 \end{split}
\]
Here we used in the second equality equations~\eqref{rod} and \eqref{eq:rodx}, next the 
inequality provided by Lemma~\ref{lem:zhou}.

The condition on the discriminant
\[
 \Delta=\alpha^2\gamma^2-2\gamma\Bigl(\alpha\beta\gamma-\textstyle\frac{3-\gamma}{2}+\delta\Bigr)\le0
\]
guarantees that the quadratic form inside the brackets is nonnegative.
Minimizing $\Delta$ with respect to~$\alpha$ gives $\alpha=\beta$. 
Choosing $\alpha=\beta$, the condition $\Delta\le0$ 
boils down to
\[
 \beta^2\ge \textstyle\frac{3-\gamma}{\gamma}-\frac{2\delta}{\gamma}=\beta_\gamma^2,
\]
where $\beta_\gamma$ is given by~\eqref{def:beta}.

In particular, choosing $\alpha=\beta=\beta_\gamma$ (this is indeed possible when $1\le\gamma\le4$ but
not when $0<\gamma<1$, 
as in the former case $0\le\beta\le1$ and Lemma~\ref{lem:zhou} can be applied) we get
\[ A_t(t,x)\ge0,\qquad\text{for all $t\in[0,T)$, $x\in\R$}.\]

Similarly, for $\alpha\in\R$ and $0\le\beta\le1$,
\[
 \begin{split}
  B_t(t,x)
   &=-e^{-\alpha q(t,x)}
   \Bigl[\Bigl(\alpha\beta\gamma-\textstyle\frac{3-\gamma}{2}\Bigl)u^2+\frac{\gamma}{2}u_x^2+\alpha\gamma uu_x +(p+\beta\partial_x p)*\rodnl
   \Bigr]\\
  &\le -e^{-\alpha q(t,x)}
  \Bigl[\Bigl(\alpha\beta\gamma-\textstyle\frac{3-\gamma}{2}+\delta\Bigl)u^2+\alpha\gamma uu_x
+\frac{\gamma}{2}u_x^2   \Bigr].
 \end{split}
\]
The condition guaranteeing that the quadratic form inside the brackets is non-negative is identical
to the previous one. 

Summarizing, we established the following fundamental proposition:
\begin{proposition}
\label{prop:liapu}
Let $1\le\gamma\le4$,  $u$ as in Theorem~\ref{th:blow}
and $\beta_\gamma$ as in~\eqref{def:beta}.
Set 
\begin{equation*}
\begin{split}
A(t,x)&=e^{\beta_\gamma q(t,x)}(\beta_\gamma u-u_x)(t,q(t,x)),\\
B(t,x)&=e^{-\beta_\gamma q(t,x)}(\beta_\gamma u+u_x)(t,q(t,x)).
\end{split}
\end{equation*}
Then, for all~$x\in\R$, 
the function $t\mapsto A(t,x)$ is monotonically increasing and $t\mapsto B(t,x)$ is decreasing.
\end{proposition}

It is then convenient to factorize
\[
(\beta_\gamma^2 u^2-u_x^2)(t,q(t,x))=A(t,x)B(t,x).
\]
This being achieved, the proof now can proceed like Constantin's \cite{ConstAIF00}:
from inequality~\eqref{inneq} we get
\begin{equation}
 \begin{split}
 \frac{\dd }{\dd t}[u_x(t,q(t,x)]
 &=\textstyle\frac{\gamma}{2}(\beta_\gamma^2 u-u_x^2)(t,q(t,x))\\
 &=\textstyle\frac{\gamma}{2}(AB)(t,x).
 \end{split}
\end{equation}

Now  let $x_0$ be such that $u'(x_0)<-\beta_\gamma|u(x_0)|$.
We denote $g(t)=u_x(t,q(t,x_0))$, $A(t)=A(t,x_0)$ and $B(t)=B(t,x_0)$.
Proposition~\ref{prop:liapu} yields, for all $t\in[0,T)$,
\[ 
A(t)\ge A(0)>0\quad\text{and}\quad B(t)\le B(0)<0
\]  
Thus, $AB(t)\le AB(0)<0$.
Then we get, for all $t\in[0,T)$,
\[
\begin{split}
g'(t)\le \textstyle\frac{\gamma}{2}AB(t)\le \frac{\gamma}{2}AB(0)<0.
\end{split}
\]

Assume, by contradiction, $T=\infty$.
Let $\alpha_0=\frac{\gamma}{2}(u'(0)^2-\beta_\gamma^2 u_0^2)(x_0)$.
Then $g(t)\le g(0)-\alpha_0t$. 
We can choose $t_0$ such that  $g(0)-\alpha_0t_0\le0$ and 
$(g(0)-\alpha_0t_0)^2\ge \|u_0\|_{H^1}^2$.
For  $t\ge t_0$, using that $g(t)\le g(t_0)$, and the conservation of the energy $E(u)=\|u(t)\|_{H^1}^2=\|u_0\|_{H^1}^2$,
 we get
\begin{equation*}
\begin{split}
g'(t)
&\le \textstyle\frac{\gamma}{2}A(t)B(t)=\frac{\gamma}{2}\bigl(\beta_\gamma^2u^2-u_x^2\bigr)(t,q(t,x_0))\\
&\le \textstyle\frac{\gamma}{2}\Bigl( \frac{\beta_\gamma^2}{2}\|u_0\|_{H^1}^2-g(t)^2\Bigr)\\
&\le -\textstyle\frac{\gamma}{4}g(t)^2.
\end{split}
\end{equation*}
This differential inequality implies that, for $t\ge t_0$,
\[
g(t)\le \frac{4g(t_0)}{4+\gamma (t-t_0)g(t_0)}.
\]
Thus, $u_x(t,q(t,x_0))$ must blow up in finite time, and $T^*\le t_0+4/(\gamma|g(t_0)|)<\infty$. This rough upper bound for $T^*$ could be slightly improved refining the above inequality for $g'(t)$.
The condition of the blowup scenario~\eqref{blowup-scena} is fulfilled.
The conclusion thus follows at least when $u_0\in H^3(\R)$. 

If $3/2<s<3$ and $u_0\in H^s(\R)$, we first approximate $u_0$ in the $H^s$-norm by a sequence of data $u_0^n$ belonging to $H^3(\R)$. By the continuous dependence on the data, we see passing to the limit as $n\to\infty$ that the above estimate for $g(t)$ and the upper bound for $T^*$ remain valid also in this case.
\end{proof}

The above proof provides some rough information on the location of the blowup: indeed from the definition of the flow map~\eqref{flow}, it follows that the formation of a singularity occurs somewhere inside the interval 
$[x_0- \frac{\gamma \|u_0\|_{H^1}T^*}{\sqrt 2}, x_0+\frac{\gamma \|u_0\|_{H^1}T^*}{\sqrt 2}]$.


\section{Earlier blow criteria for the rod equation. Some generalizations.}
\label{sec-erod}

Experimental studies on a few hyper-elastic materials revealed that the parameter $\gamma$ can range from
-29.4770 to 3.4174, see \cite{Dai-Huo}. The interval $[1,4]\ni\gamma$ where  Theorem~\ref{th:blow} can be applied thus overlap, but not completely cover, the physically interesting cases.

Adapting to the elastic rod equation the proof of~Theorem~\ref{th:CE},
Y. Zhou~\cite{Zhou-Math-Nachr} established that for $0<\gamma<3$ a blowup occurs provided $\exists\,x_0\in\R$ such that 
\begin{equation}
\label{zhocr}
 |u_0'(x_0)|<-\textstyle\frac{\beta_\gamma}{\sqrt 2}\|u_0\|_{H^1}.
\end{equation}
Theorem~\ref{th:blow}  thus improves Zhou's criterion in the range $1\le\gamma<3$
and extends it to $3\le\gamma\le 4$.
On the other hand, Theorem~\ref{th:blow} cannot be applied for $0<\gamma<1$.
For reader's convenience, let us give a direct and simple proof of  criterion~\eqref{zhocr},
valid for $0<\gamma\le 4$.
We only need to recall inequality~\eqref{inneq}
\begin{equation*}
\frac{\dd }{\dd t}[u_x(t,q(t,x))]
 \le\textstyle\frac{\gamma}{2}(\beta_\gamma^2 u^2-u_x^2)(t,q(t,x)),
\end{equation*}
and the usual $L^\infty$-estimate
\[
\|u\|_\infty\le \textstyle\frac{1}{\sqrt 2}\|u\|_{H^1}=\frac{1}{\sqrt 2}\|u_0\|_{H^1}.
\] 
Letting $g(t)=u_x(t,q(t,x_0))$, then we obtain the differential inequality
\begin{equation*}
g'(t)\le \textstyle\frac{\gamma}{2}\Bigl( \frac{\beta_\gamma^2}{2}\|u_0\|_{H^1}^2-g(t)^2\Bigr)
\end{equation*}
with the initial condition $g(0)=u_0'(x_0)$.
Condition~\eqref{zhocr} ensures $g'(0)<0$ and the blowup follows for $0<\gamma\le4$
integrating this Riccati-type differential inequality .

\begin{remark}
The reason for excluding $\gamma<1$ in Theorem~\ref{th:blow} is that
one would have $\beta_\gamma>1$ (see Figure~\ref{figure1}). In this case, it is not difficult to see that the result of Proposition~\ref{prop:liapu} (implying that the condition $u_0'(x_0)<-\beta_\gamma|u_0(x_0)|$ is conserved by the flow) is no longer valid. The main reason is the lack of the inequalities 
\[
p\pm\beta\partial_x p\ge0
\]
when $\beta>1$.

According to \cite{ConStra00} and \cite{Dai-Huo}, equation~\eqref{rod} admits smooth and orbitally stable solitary waves solutions of the form $\phi_c(x-ct)$ for $\gamma<1$ and orbitally stable peaked solitons $ce^{-|x-ct|}$ for $\gamma=1$.
As observed in~\cite{ConStra00}, if $u(t,x)=\phi_c(x-ct)$ is a solution vanishing at infinity, then the profile $\phi_c$ must be a translate
of $c\phi$, i.e.,
\[
\phi_c(x)=c\phi(x-a), \qquad x\in\R,
\]
where $\phi$ solves the differential equation
\[
(\phi')^2=\gamma\phi(\phi')^2-\phi^3+\phi^2.
\] 
Moreover, $\phi$ is positive, even and monotonically decreasing from its peak.
In addition, it decays exponentially as $|x|\to\infty$, and the same for its derivatives.
For $\gamma<1$, the initial datum $u_0(x)=c\phi(x)$ gives rise to a global smooth solution $u_c=c\phi(x-ct)$.
It follows from the last equation, the parity  and the decay properties of~$\phi$, that we can find a large enough~$x_0$ such that $u'(x_0,t)\simeq -|u(x_0,t)|$.
This means that if one believes in the validity of a local-in-space blowup criterion of the form
 $u_0'(x_0)<-\alpha_\gamma|u_0(x_0)|$ also for $\gamma<1$, then $\alpha_\gamma$ 
 must satisfy $\alpha_\gamma\ge1$.

The restriction $\gamma>4$  might be purely technical as well.
Of course, the expression of the coefficient $\beta_\gamma$ computed in~\eqref{def:beta} in this case does not make sense, but one might conjecture that the statement of~Theorem~\ref{th:blow} holds with another coefficient~$\alpha_\gamma$.
\end{remark}

Let us stress the fact that for $\gamma<0$  or $\gamma>4$ nonlocal blowup criteria are still available:
for example,  it was observed in \cite{ConStra00}, \cite{Zhou-Math-Nachr} that 
a blowup occurs when, for some $x_0$,
\[
u_0'(x_0)>\textstyle\frac{\sqrt{(\gamma-3)/\gamma}}{\sqrt2}\|u_0\|_{H^1}, \qquad\text{if $\gamma<0$},
\]
or
\[
u_0'(x_0)<-\textstyle\frac{\sqrt{(\gamma-3)/\gamma}}{\sqrt2}\|u_0\|_{H^1}, \qquad\text{if $\gamma>3$}.
\]
This readily follows from equation~\eqref{eq:rodx}.
However, for $3<\gamma\le4$, criterion~\eqref{zhocr} is slightly better, as 
$\beta_\gamma\le\sqrt{(\gamma-3)/\gamma}$.

\medskip
In this paper we only considered vanishing boundary conditions at infinity. Different kind of boundary conditions are also of interest.
See~\cite{GHR2012}.
Our method could be used to obtain blowup results, {\it e.g.\/}, for solutions with constant limit at $x\to+\infty$ and $x\to-\infty$,

The main result of the present paper remains valid, after suitable modifications, to other one dimensional models for the propagation of nonlinear waves. For example,  equation~\eqref{rod-pde} is a particular case
(corresponding to $f(u)=\frac{\gamma-3}{2} u^2$, $g(u)=\gamma u^2$) of the so called 
``generalized hyperelastic rod equation'',  
\[
 u_t+f'(u)u_x+\partial_xp*\bigl[g(u)+\textstyle\frac{f''(u)}{2}u_x^2\bigr]=0,                                            
                                            \]
studied, {\it e.g.\/}, in~\cite{HolJDE07}.
The choice $f(u)=u^2$ and $g(u)=\kappa u+u^2$, with $\kappa\ge0$, would correspond to the Camassa--Holm equation with the additional term $\kappa u_x$.
In a forthcoming paper we will address the blowup issue for a large class of functions $f$ and $g$.

\section{Acknowledgements}
The author is grateful to the referee for his careful reading and for his interesting suggestions on the possible developements of the present study.

\end{document}